\documentclass [11pt,a4paper]{amsart}

\usepackage{amssymb} 
\usepackage{amsfonts} 
\usepackage{amsmath}
\usepackage{amsthm} 
\usepackage{tikz}
\usepackage[utf8]{inputenc}
\usepackage[all]{xy}
\SelectTips{cm}{}

\usepackage[pagebackref,
    ,pdfborder={0 0 0}
    ,urlcolor=black,a4paper,hypertexnames=false]{hyperref}
\hypersetup{pdfauthor={Clara L\"oh},pdftitle=Explicit l1-efficient cycles}

\usepackage{caption}

\newtheorem{lem}{Lemma}[section]
\newtheorem{thm}[lem]{Theorem}
\newtheorem{prop}[lem]{Proposition}
\newtheorem{cor}[lem]{Corollary} 

\theoremstyle{definition}

\newtheorem{question}[lem]{Question}

\newtheorem{rem}[lem]{Remark} 

\newcommand{\N}{\ensuremath {\mathbb{N}}}
\newcommand{\R} {\ensuremath {\mathbb{R}}}
\newcommand{\Q} {\ensuremath {\mathbb{Q}}}

 \makeatletter
 \newcommand\norm{\bBigg@{0.8}}

 \makeatother
 \newcommand{\indnorm}[2][flex]{\csname #1l\endcsname\|#2%
                                 \csname #1r\endcsname\|\mathclose{}}

\newcommand{\ucov}[1]{%
  \widetilde{#1}}

\DeclareMathOperator{\map}{map}
\DeclareMathOperator{\id}{id}

\title[Explicit $\ell^1$-efficient cycles]%
      {Explicit $\ell^1$-efficient cycles\\ and amenable normal subgroups}

\author{Clara L\"oh}

\subjclass[2010]{57R19, 55N99, 20F65}%

\keywords{simplicial volume, amenable groups}

\def\draftinfo{}
\date{\today.\ \copyright{\ C.~L\"oh 2017}. 
    This work was supported by the CRC~1085 \emph{Higher Invariants} 
    (Universit\"at Regensburg, funded by the DFG)\draftinfo}

\begin{document}

\begin{abstract}
  By Gromov's mapping theorem for bounded cohomology, the projection
  of a group to the quotient by an amenable normal subgroup is
  isometric on group homology with respect to the $\ell^1$-semi-norm.
  Gromov's description of the diffusion of cycles also implicitly
  produces efficient cycles in this situation. We present an elementary 
  version of this explicit construction.
\end{abstract}
\maketitle

\section{Introduction}

The size of classes in group homology with coefficients in a normed
module can be measured by the $\ell^1$-semi-norm
(Section~\ref{sec:l1}). This semi-norm admits a description in terms
of bounded cohomology~\cite{vbc}. 

By Gromov's mapping theorem for bounded cohomology, the
projection~$\pi \colon \Gamma \longrightarrow \Gamma/N$ of a group~$\Gamma$ to
the quotient by an amenable normal subgroup~$N$ induces an isometric
isomorphism in bounded cohomology~\cite{vbc,ivanov}. In particular,
the map induced by~$\pi$ on homology (e.g., with real coefficients) is
isometric with respect to the $\ell^1$-semi-norm. More concretely,
Gromov's description of diffusion of cycles also implicitly produces
$\ell^1$-efficient cycles in this situation; similar diffusion
arguments have been used in applications to Morse theory~\cite{alpertkatz}
and Lipschitz simplicial volume~\cite{strzalkowski}. In this note, we
present a slightly simplified version of this explicit construction.
Because we focus on the group homology case, we can avoid the use of
multicomplexes or model complexes for cycles. More precisely, we
consider averaging maps of the following type to produce
$\ell^1$-efficient cycles:

\begin{thm}\label{mainthm}
  Let $\Gamma$ be a group, let $N \subset \Gamma$ be a finitely
  generated amenable normal subgroup, let $n \in \N$, let $c \in
  C_n(\Gamma;\R)$ be a cycle, and let $(F_k)_{k \in \N}$ be a F\o lner
    sequence for~$N$. Then the cycles~$\psi_k(c)$ are homologous to~$c$ and 
  \[ \lim_{k\rightarrow \infty} \bigl| \psi_k(c)\bigr|_1 = |\overline c|_1.
  \]
  Here, $\overline c \in C_n(\Gamma/N;\R)$ denotes the push-forward
  of~$c$ under the canonical projection~$\Gamma \longrightarrow
  \Gamma/N$ and for~$k \in \N$, the averaging map~$\psi_k$ is
  defined by
  \begin{align*}
    \psi_k \colon C_n(\Gamma;\R)
    & \longrightarrow C_n(\Gamma;\R)
    \\
      [\gamma_0, \dots, \gamma_n]
      & \longmapsto \frac 1{|F_k|^{n+1}} \cdot \sum_{\eta \in F_k^{n+1}}
      [\gamma_0\cdot \eta_0, \dots, \gamma_n \cdot \eta_n]
  \end{align*}
\end{thm}

Standard arguments then let us derive the following statement,
which originally is a simple consequence of the mapping theorem
in bounded cohomology~\cite{vbc}:

\begin{cor}\label{cor:isometric}
  Let $\Gamma$ be a group, let $N \subset \Gamma$ be a finitely
  generated amenable normal subgroup and let $\pi \colon \Gamma
  \longrightarrow \Gamma/N$ be the canonical projection. Then the
  map~$H_*(\pi;\R) \colon H_*(\Gamma;\R) \longrightarrow
  H_*(\Gamma/N;\R)$ is isometric with respect to the
  $\ell^1$-semi-norm.
\end{cor}

More explicitly, in the situation of Corollary~\ref{cor:isometric}, we 
can construct $\ell^1$-efficient cycles out of knowledge on
$\ell^1$-efficient cycles for~$H_*(\Gamma/N;\R)$ and a F\o lner
sequence of~$N$.  This procedure is explained in detail in
Section~\ref{sec:exp}.

\subsection{Twisted coefficients}

Our proof of Theorem~\ref{mainthm} and Corollary~\ref{cor:isometric}
also carries over to general normed coefficients.

\begin{thm}\label{mainthmcoeff}
  Let $\Gamma$ be a group, let $N \subset \Gamma$ be a finitely
  generated amenable normal subgroup, let $A$ be a normed $\R\Gamma$-module,
  let $n \in \N$, let $c \in C_n(\Gamma;A)$ be a cycle, and let $(F_k)_{k \in \N}$
  be a F\o lner sequence for~$N$. Then the cycles~$\psi_k(c)$ are homologous to~$c$
  and
  \[ \lim_{k \rightarrow \infty} \bigl|\psi_k(c)\bigr|_1 = |\overline c|_1.
  \]
  Here, $\overline c \in C_n(\Gamma/N;A_N)$ denotes the push-forward of~$c$ under
  the canonical projection~$\Gamma\longrightarrow \Gamma/N$ and for~$k \in \N$,
  the averaging map~$\psi_k$ is defined by
  \begin{align*}
    \psi_k \colon C_n(\Gamma;A) & \longrightarrow C_n(\Gamma;A) \\
    a \otimes (\gamma_0, \dots, \gamma_n)
    & \longmapsto
    \frac1{|F_k|^{n+1}} \cdot \sum_{\eta \in F_k^{n+1}} a \otimes (\gamma_0\cdot \eta_0, \dots, \gamma_n\cdot \eta_n).
  \end{align*}
\end{thm}

\begin{cor}\label{cor:isometriccoeff}
  Let $\Gamma$ be a group, let $N\subset \Gamma$ be a finitely
  generated amenable normal subgroup, let $\pi \colon \Gamma \longrightarrow \Gamma/N$
  be the canonical projection, and let $A$ be a normed~$\R \Gamma$-module. Then
  the map~$H_*(\pi;A) \colon H_*(\Gamma;A) \longrightarrow H_*(\Gamma/N;A_N)$
  is isometric with respect to the $\ell^1$-semi-norm.
\end{cor}

Originally, Corollary~\ref{cor:isometriccoeff} was proved by
Ivanov~\cite{ivanov} via an algebraic approach to bounded
cohomology. Again, our proof gives a recipe to construct
$\ell^1$-efficient cycles out of knowledge on~$\Gamma/N$ and a F\o
lner sequence of~$N$.

\subsection{Simplicial volume}

Simplicial volume of an oriented closed connected manifold is defined
as the $\ell^1$-semi-norm of the $\R$-fundamental class~\cite{vbc}. As
simplicial volume is related to both Riemannian geometry and mapping
degrees, one is interested in calculations of simplicial volume and
the construction of explicit $\ell^1$-efficient $\R$-fundamental cycles. 
The methods for group homology have a canonical counterpart for
aspherical spaces and hence for simplicial volume of aspherical
manifolds (Section~\ref{subsec:asph}).

The geometric F\o lner
fillings used in the context of integral foliated simplicial volume of
aspherical manifolds with amenable fundamental
group~\cite{FLPS,fauserloeh} are related to the averaging maps above
but are not quite the same. The F\o lner construction in the present
article does not use the deck transformation action of the fundamental
group, but a slightly different action by tuples on the chain
level. As a consequence, it is not clear whether the proof of
Theorem~\ref{mainthm} can be adapted to the integral foliated
simplicial volume or stable integral simplicial volume setting. The
wish to understand $\ell^1$-efficient cycles of aspherical manifolds
in the presence of amenable normal subgroups derives from
the following open problem:

\begin{question}[L\"uck]
  Let $M$ be an oriented closed connected aspherical manifold
  with the property that $\pi_1(M)$ contains an infinite amenable
  normal subgroup. Does this imply~$\|M\| = 0$\;?
\end{question}

\subsection*{Organisation of this note}

In Section~\ref{sec:l1}, we recall the $\ell^1$-semi-norm on group
homology and singular homology. Section~\ref{sec:avg} contains the
basic averaging argument and the proof of Theorems~\ref{mainthm} and
Theorem~\ref{mainthmcoeff}.  In Section~\ref{sec:exp}, we discuss how
Theorem~\ref{mainthm} can be used to give explicit $\ell^1$-efficient
cycles in group homology (which proves Corollaries~\ref{cor:isometric}
and \ref{cor:isometriccoeff}) and singular homology.

\section{The $\ell^1$-semi-norm}\label{sec:l1}

We recall the definition of the $\ell^1$-semi-norm on group homology
and singular homology.

\subsection{Normed modules and $\ell^1$-norms}\label{subsec:l1def}

Let $\Gamma$ be a group. Then a \emph{normed $\Gamma$-module} is an $\R
\Gamma$-module~$A$ together with a (semi-)norm such that the $\Gamma$-action
on~$A$ is isometric. If $X$ is a free $\Gamma$-set, then the free
$\R\Gamma$-module~$\bigoplus_X \R$ is a normed $\R
\Gamma$-module with respect to the $\ell^1$-norm associated to the
canonical $\R$-basis~$(e_x)_{x \in X}$. If $A$
is a normed $\Gamma$-module, then we define the \emph{$\ell^1$-norm}
on the $\R$-vector space~$A \otimes_\Gamma \bigoplus_X \R$ by
\begin{align*}
  |\cdot|_1 \colon 
  A \otimes_\Gamma \bigoplus_X \R & \longrightarrow \R_{\geq 0} \\
  \sum_{x\in X} a_{x} \otimes e_x
  & \longmapsto \sum_{\Gamma\cdot x \in \Gamma \setminus X} \biggl| \sum_{\gamma \in \Gamma} \gamma^{-1} \cdot a_{\gamma \cdot x}\biggr|.
\end{align*}
Here, in order to define the tensor product~$A \otimes_\Gamma
\bigoplus_X \R$, we implicitly converted the (left)
$\R\Gamma$-module~$A$ to a right $\R\Gamma$-module via the involution
on the group ring~$\R \Gamma$ given by inversion of group elements.

\begin{rem}[rational coefficients]
  All of our arguments also work in the case of normed $\Q\Gamma$-modules;
  however, for the sake of simplicity, we will formulate everything in
  the more standard case of normed $\R\Gamma$-modules. 
\end{rem}

\subsection{The $\ell^1$-semi-norm: groups}

Let $\Gamma$ be a group. We write~$C_*(\Gamma)$ for the simplicial
$\R\Gamma$-resolution of~$\R$. For~$n \in \N$, we consider the
\emph{$\ell^1$-norm~$|\cdot|_1$} on~$C_n(\Gamma) =
\bigoplus_{\Gamma^{n+1}}\R$, where $\Gamma$ acts diagonally
on~$\Gamma^{n+1}$. If $A$ is a normed $\R\Gamma$-module, then
$C_n(\Gamma;A) := A \otimes_\Gamma C_n(\Gamma)$ carries a natural
$\ell^1$-norm, as described in Section~\ref{subsec:l1def}.  This norm
induces the \emph{$\ell^1$-semi-norm} on group homology~$H_n(\Gamma;A) :=
H_n(A \otimes_\Gamma C_*(\Gamma))$:
\begin{align*}
  \|\cdot\|_1 \colon H_n(\Gamma;A) & \longrightarrow \R_{\geq 0} \\
  \alpha & \longmapsto \inf\bigl\{ |c|_1 \bigm| c \in C_n(\Gamma;A),\ \partial c = 0, [c] = \alpha\bigr\}.
\end{align*}

If $N \subset \Gamma$ is a normal subgroup of~$\Gamma$ and $A$ is a
normed $\Gamma$-module, then the $N$-coinvariants~$A_N$ of~$A$ form a
normed $\Gamma/N$-module (with the quotient semi-norm), and
the canonical projections~$\pi \colon \Gamma \longrightarrow \Gamma/N$
and $A \longrightarrow A_N$ induce a well-defined chain
map~$C_*(\pi;A)\colon C_*(\Gamma;A) \longrightarrow
C_*(\Gamma/N;A_N)$. In particular, we obtain an induced
homomorphism~$H_*(\pi;A) \colon H_*(\Gamma;A) \longrightarrow
H_*(\Gamma/N;A_N)$. These are the maps occuring in
Theorem~\ref{mainthmcoeff} and Corollary~\ref{cor:isometriccoeff}.

\subsection{The $\ell^1$-semi-norm: spaces}

Let $M$ be a path-connected topological space that admits a universal
covering~$\ucov M$ with deck transformation action by the fundamental
group~$\Gamma := \pi_1(M)$ (e.g., $M$ could be a connected manifold or
CW-complex). If $n \in \N$, then the singular chain module~$C_n(\ucov
M;\R)$ is a free $\R\Gamma$-space, as witnessed by the
$\R$-basis~$\map(\Delta^n,\ucov M)$, which is a free $\Gamma$-set with
respect to the deck transformation action. Hence, $C_n(\ucov M;\R)$ is
a normed $\Gamma$-module with respect to the $\ell^1$-norm. If $A$ is
a normed $\Gamma$-module, then the twisted chain module~$C_n(M;A) := A
\otimes_\Gamma C_n(\ucov M;\R)$ inherits an $\ell^1$-norm (see
Section~\ref{subsec:l1def}). In the case that $A$ is trivial
$\Gamma$-module~$\R$, this description coincides with the classical
$\ell^1$-norm on~$C_n(M;\R)$. The semi-norm on~$H_*(M;A)$ induced by
the $\ell^1$-norm on~$C_*(M;A)$ is the \emph{$\ell^1$-semi-norm on
  singular homology}. If $M$ is an oriented closed connected $n$-manifold,
then
\[ \| M\| := \bigl\| [M]_\R \bigr\|_1 
\]
is Gromov's \emph{simplicial volume of~$M$}~\cite{vbc,mapsimvol}, where $[M]_\R
\in H_{\dim M}(M;\R)$ denotes the $\R$-fundamental class of~$M$.

In the aspherical case, this $\ell^1$-semi-norm on homology coincides
with the $\ell^1$-semi-norm on group homology~\cite{vbc,ivanov}. More
explicitly:

\begin{prop}\label{prop:comp}
  Let $M$ be a path-connected topological space that admits a weakly
  contractible universal covering with deck transformation action by
  the fundamental group~$\Gamma := \pi_1(M)$, and let $D \subset \ucov M$
  be a set-theoretic fundamental domain of the deck transformation action. Then
  the $\Gamma$-equivariant chain map~$C_*(\ucov M;\R) \longrightarrow C_*(\Gamma)$
  given by
  \begin{align*}
    C_n(\ucov M;\R) & \longrightarrow C_n(\Gamma;\R) \\
    \sigma & \longmapsto (\gamma_0, \dots, \gamma_n) \text{ with $\sigma(e_j) \in \gamma_j \cdot D$}
  \end{align*}
  admits a $\Gamma$-equivariant chain homotopy inverse (given by
  inductive filling of simplices). In particular, if $A$ is a normed
  $\Gamma$-module, then these chain maps induce mutually inverse
  isometric isomorphisms~$H_*(M;A) \cong H_*(\Gamma;A)$. \qed
\end{prop}

Actually, even more is known to be true: The classifying map~$M
\longrightarrow B\Gamma$ is known to induce an isometric map on
singular homology with respect to the $\ell^1$-semi-norm~\cite{vbc,ivanov}.

\subsection{Survey of known explicit efficient cycles}

Explicit constructions of $\ell^1$-efficient cycles are known for
spheres and tori~\cite{vbc}, cross-product classes (through the
homological cross-product and the duality principle)~\cite{vbc},
finite coverings (through lifts of cycles)~\cite{vbc,mapsimvol}, hyperbolic
manifolds (through straightening/smearing)~\cite{benedettipetronio}, manifolds with non-trivial
$S^1$-action (through an inductive filling/wrapping
construction)~\cite{yano}, spaces with amenable fundamental group
(through diffusion or F\o lner filling)~\cite{vbc,alpertkatz,fauserloeh},
and for some amenable glueings~\cite{vbc,kuessner}.

\section{Averaging chains}\label{sec:avg}

\subsection{Finite normal subgroups}

We will start by giving some context. Namely, we will first recall
the case of finite normal subgroups; we will then explain how F\o lner
sequences allow to pass from the finite case to the amenable case. 
We therefore recall the following well-known statement (whose proof
is a straightforward calculation):

\begin{prop}
  Let $\Gamma$ be a group, let $N \subset \Gamma$ be a finite normal
  subgroup, and let $\pi \colon \Gamma \longrightarrow \Gamma/N$ be
  the projection map. Then
  \begin{align*}
    C_n(\Gamma/N;\R) & \longrightarrow C_n(\Gamma;\R) \\
    [\gamma_0 \cdot N, \dots, \gamma_n\cdot N]
    & \longmapsto
    \frac1{|N|^{n+1}} \cdot
    \sum_{\eta \in N^{n+1}} [\gamma_0\cdot \eta_0, \dots, \gamma_n \cdot \eta_n]
  \end{align*}
  defines a well-defined chain map~$C_*(\Gamma/N;\R) \longrightarrow
  C_*(\Gamma;\R)$ that induces the inverse of~$H_*(\pi;\R) \colon
  H_*(\Gamma;\R) \longrightarrow H_*(\Gamma/N;\R)$ on homology.
\end{prop}
  
\subsection{Averaging maps}

In order to pass from the case of finite normal subgroups~$N$ to
larger amenable normal subgroups, we replace the averaging over
the subgroup~$N$ by averaging over F\o lner sets.  
If the normal subgroup~$N \subset \Gamma$ is infinite, in general,
there does not exist a splitting of the homomorphismus~$H_*(\Gamma;\R)
\longrightarrow H_*(\Gamma/N;\R)$ induced by the canonical
projection~$\Gamma \longrightarrow \Gamma/N$ and if $F \subset
N$ is finite, 
\begin{align*}
  C_n(\Gamma/N;\R) & \longrightarrow C_n(\Gamma;\R) \\
  [\gamma_0 \cdot N, \dots, \gamma_n \cdot N]
  & \longmapsto
  \frac1{|F|^{n+1}} \cdot \sum_{\eta \in F^{n+1}}
  [\gamma_0\cdot \eta_0, \dots, \gamma_n \cdot \eta_n] 
\end{align*}
in general will \emph{not} produce a well-defined map. However, the
following ``round-trip'' averaging map over finite subsets is always
well-defined.

\begin{prop}[averaging over a finite subset]\label{prop:avg}
  Let $\Gamma$ be a group and let $F \subset \Gamma$ be a non-empty finite
  subset. Then the map~$\psi^F_* \colon C_*(\Gamma) \longrightarrow
  C_*(\Gamma)$ defined by
  \begin{align*}
    \psi^F_n \colon C_n(\Gamma)& \longrightarrow C_n(\Gamma) \\
    (\gamma_0,\dots,\gamma_n) & \longmapsto
    \frac1{|F|^{n+1}} \cdot \sum_{\eta \in F^{n+1}} (\gamma_0 \cdot \eta_0, \dots,
    \gamma_n \cdot \eta_n)
  \end{align*}
  has the following properties:
  \begin{enumerate}
    \item The map~$\psi^F_*$ is $\Gamma$-equivariant.
    \item The map~$\psi^F_*$ is a chain map.
    \item The chain map~$\psi^F_*$ is $\Gamma$-equivariantly
      chain homotopic to the identity.
    \item We have~$\|\psi^F_*\| \leq 1$ (with respect to the
      $\ell^1$-norm).
  \end{enumerate}
\end{prop}
\begin{proof}
  Because left and right multiplication on~$\Gamma$ do not interfere
  with each other, the map~$\psi_*^F$ is
  $\Gamma$-equivariant. Moreover, by definition of the $\ell^1$-norm,
  we clearly have~$\|\psi^F_*\| \leq 1$.

  We now show that $\psi^F_*$ is a chain map: Let $n \in \N$ and
  $(\gamma_0, \dots, \gamma_n) \in C_n(\Gamma)$. By construction,
  we have
  \begin{align*}
    & \phantom{=}\ 
    \partial\psi^F_n\bigl((\gamma_0, \dots, \gamma_n)\bigr)
    \\
    & = \sum_{j=0}^n (-1)^j \cdot \frac1{|F|^{n+1}}
    \cdot \sum_{\eta \in F^{n+1}} (\gamma_0 \cdot \eta_0, \dots, \gamma_{j-1}\cdot \eta_{j-1}, \gamma_{j+1} \cdot \eta_{j+1}, \dots, \gamma_n \cdot \eta_n)
  \end{align*}
  and reindexing shows (the last summand does not depend on~``$\eta_j$'') 
  \begin{align*}
    & \phantom{=}\ 
    \partial\psi^F_n\bigl((\gamma_0, \dots, \gamma_n)\bigr)
    \\
    & = \sum_{j=0}^n (-1)^j \cdot \frac{|F|}{|F|^{n+1}}
    \cdot \sum_{\eta \in F^{n}} (\gamma_0 \cdot \eta_0, \dots, \gamma_{j-1}\cdot \eta_{j-1}, \gamma_{j+1} \cdot \eta_{j}, \dots, \gamma_n \cdot \eta_{n-1})   
    \\
    & = \frac1{|F|^n} \cdot \sum_{\eta \in F^n} \sum_{j=0}^n
    (-1)^j \cdot (\gamma_0 \cdot \eta_0, \dots, \gamma_{j-1} \cdot \eta_{j-1}, \gamma_{j+1}\cdot \eta_j, \dots, \gamma_n \cdot \eta_{n-1})
    \\
    & =  
    \psi^F_n\bigl(\partial(\gamma_0, \dots, \gamma_n))\bigr).
  \end{align*}
  
  It remains to prove the third part: Because $C_*(\Gamma)$ is a free
  $\R\Gamma$-resolution of the trivial $\Gamma$-module~$\R$ and
  because the chain map~$\psi_*^F$ extends the identity~$\id_\R \colon
  \R \longrightarrow\R$ of the resolved module, the fundamental lemma
  of homological algebra immediately implies that $\psi^F_*$ is
  $\Gamma$-equivariantly chain homotopic to the identity.  
\end{proof}

Passing to the tensor product, we obtain corresponding chain
maps on the standard complexes: 

\begin{cor}\label{cor:avg}
  Let $\Gamma$ be a group, let $A$ be a normed $\R \Gamma$-module, and
  let $F \subset \Gamma$ be a non-empty finite subset. Then 
  $\id_A \otimes_\Gamma \psi_*^F \colon C_*(\Gamma;A)
     \longrightarrow C_*(\Gamma;A)
  $
  is a well-defined chain map that is chain homotopic to the identity
  and that satisfies
  $\| \id_A \otimes_\Gamma \psi^F_* \| \leq 1.
  $ \qed
\end{cor}

\subsection{The push-forward estimate}

As final ingredient, we estimate the norm of averaged chains in terms
of the image in the quotient group. For subsets~$F,S \subset N$ of a group~$N$, we use the
version
\[ \partial_S F := \{ \eta \in F \mid \exists_{\sigma \in S \cup S^{-1}}\ \sigma \cdot\eta \not \in F\}
\]
of the \emph{$S$-boundary of~$F$ in~$N$}.

\begin{prop}[the push-forward estimate]\label{prop:normestimate}
  Let $\Gamma$ be a group, let $A$ be a normed $\R \Gamma$-module, let
  $N \subset \Gamma$ be a normal subgroup, and let $c \in C_n(\Gamma;A)$.
  For every~$\varepsilon \in \R_{>0}$ there exists a finite set~$S \subset N$
  and a constant~$K \in \R_{>0}$ wit the following property: For all non-empty finite
  subsets~$F \subset N$ we have
  \[ \bigl| (\id_A \otimes_\Gamma \psi^F_n)(c)\bigr|_1
     \leq |\overline c|_1 + \varepsilon + K \cdot \frac{|\partial_S F|}{|F|} 
  \]
  Here, $\overline c \in C_*(\Gamma/N;A_N)$ denotes the push-forward of~$c$ along the projection~$\Gamma \longrightarrow \Gamma/N$.
\end{prop}

\begin{proof}
  Let $\varepsilon \in \R_{>0}$.  Using the canonical isometric
  isomorphism between the chain module~$A \otimes_\Gamma
  C_n(\Gamma/N)$ (equipped with the semi-norm induced by the
  $\ell^1$-norm) and $C_n(\Gamma/N;A_N)$, we see that we can decompose
  \[ c = c_0 + c_1
  \]
  with~$c_0, c_1 \in C_n(\Gamma;A)$ and $\overline{c_0} = \overline c$ as
  well as $|c_0|_1 \leq |\overline c|_1 + \varepsilon$. Because the averaging
  maps of Proposition~\ref{prop:avg} are linear and have norm at most~$1$,
  we may assume without loss of generality that $\overline c = 0$. Then the
  canonical isomorphism (where $N^{n+1}$ acts by component-wise multiplication
  from the right)
  \[ C_n(\Gamma/N;A_N) \cong \bigl(A \otimes_\Gamma C_n(\Gamma)\bigr)_{N^{n+1}} 
  \]
  shows that we can write~$c$ in the form
  \[ c = \sum_{j=1}^m \bigl( c(j) - c(j) \cdot \sigma(j)\bigr)
  \]
  with $m\in \N$, $c(1), \dots, c(m) \in C_n(\Gamma;A)$, and
  $\sigma(1), \dots, \sigma(m) \in N^{n+1}$. We then set
  \begin{align*}
    S & := \bigl\{ \sigma(j)_k \bigm| j \in \{1,\dots,m\},\ k \in \{0,\dots,n\}\bigr\}
    \\
    K & := 2 \cdot (n+1) \cdot \max_{j \in \{1,\dots,m\}} \bigl| c(j)\bigr|_1. 
  \end{align*}
  By definition of the $\ell^1$-norm on~$C_n(\Gamma;A)$, we may now assume
  that $c$ is of the form
  \[ c = a \otimes (\gamma_0, \dots, \gamma_n) - a \otimes (\gamma_0 \cdot \sigma_0, \dots,
     \gamma_n \cdot \sigma_n)
  \]
  with $a \in A$ and $\gamma_0, \dots, \gamma_n \in \Gamma$, $\sigma_0, \dots, \sigma_n \in N$.    
  Then the definition of the $S$-boundary shows that for all non-empty finite subsets~$F \subset N$
  we have
  \begin{align*}
    \bigl|(\id_A \otimes_\Gamma \psi_n^F)(c)\bigr|_1
    & = \frac1{|F|^{n+1}} \cdot \biggl|a \otimes
    \biggl(\sum_{\eta \in F^{n+1}} (\gamma_0\cdot \eta_0, \dots, \gamma_n \cdot \eta_n)\\
    & \qquad\qquad\qquad- \sum_{\eta \in F^{n+1}} (\gamma_0 \cdot \sigma_0\cdot \eta_0, \dots,
    \gamma_n \cdot \sigma_n \cdot \eta_n)\biggr)\biggr|_1
    \\
    & \leq
    2 \cdot (n+1) \cdot |a| \cdot \frac{|\partial_S F|}{|F|}.
  \end{align*}
  This gives the desired estimate.
\end{proof}

\begin{rem}
  In the case of trivial $\R$-coefficients, the same argument
  shows: For every chain~$c \in C_n(\Gamma;\R)$
  there is a finite set~$S \subset N$ with the following property: For
  all non-empty finite subsets~$F \subset N$ we have
  \[ \bigl|(\id_\R \otimes_\Gamma \psi_n^F)(c)\bigr|_1
     \leq |\overline c|_1 + 2 \cdot (n+1) \cdot \frac{|\partial_S F|}{|F|}.
  \]
\end{rem}

\subsection{Proof of Theorems~\ref{mainthm} and~\ref{mainthmcoeff}}

As Theorem~\ref{mainthm} is a special case of
Theorem~\ref{mainthmcoeff}, it suffices to prove the latter:

\begin{proof}[Proof of Theorem~\ref{mainthmcoeff}]
  Let $\varepsilon \in \R_{>0}$ and let $S \subset N$ be a finite set
  that is adapted to~$c$ and $\varepsilon$ and let $K \in \R_{>0}$ be
  a constant as provided by Proposition~\ref{prop:normestimate}.  Let
  $k \in \N$.  By construction, we have $\psi_k = \id_A \otimes_\Gamma
  \psi^{F_k}_n$ (where $\psi^{F_k}_n$ is defined in
  Proposition~\ref{prop:avg}).  In view of Corollary~\ref{cor:avg},
  $c_k := \psi_k(c)$ is a cycle that represents~$[c] \in
  H_n(\Gamma;A)$. Moreover, Proposition~\ref{prop:normestimate} shows
  that
  \[ |c_k|_1 \leq |\overline c|_1 + \varepsilon + K \cdot \frac{|\partial_S F_k|}{|F_k|}.
  \]
  Because~$(F_k)_{k \in \N}$ is a F\o lner sequence for~$N$ we have
  $\lim_{k \rightarrow \infty} |\partial_S F_k|/|F_k| = 0$. Therefore,
  $\lim_{k \rightarrow \infty} |c_k|_1 \leq  |\overline c|_1 + \varepsilon$.
  Letting $\varepsilon$ go to~$0$ gives the inequality $\lim_{k \rightarrow \infty}|c_k|_1 \leq |\overline c|_1$.

  The converse inequality holds because~$\|\psi_k\|\leq 1$ and $|\overline c|_1\leq |c|_1$. 
\end{proof}

\begin{rem}\label{rem:nonfingen}
  This argument also generalises to the case that the amenable normal
  subgroup is not finitely generated; one just has to first take a
  finite set~$S \subset N$ that is adapted to the cycle in question
  and then to pick a F\o lner sequence for the finitely generated
  amenable subgroup~$\langle S\rangle_N$ of~$N$.
\end{rem}

\section{Explicit $\ell^1$-efficient cycles}\label{sec:exp}

\subsection{Efficient cycles: groups}

We can now prove Corollary~\ref{cor:isometriccoeff} (which implies
Corollary~\ref{cor:isometric}). In fact, we will explain how one
can produce explicit $\ell^1$-efficient cycles (using the input
specified in the proof).

\begin{proof}[Proof of Corollary~\ref{cor:isometriccoeff}]
  Clearly, the homomorphism~$H_*(\pi;A) \colon H_*(\Gamma;A)
  \longrightarrow H_*(\Gamma/N;A_N)$ satisfies~$\|H_*(\pi;A)\| \leq
  1$. Thus, it suffices to prove that the $\ell^1$-semi-norm cannot
  decrease under~$H_*(\pi;A)$: 

  Let $\alpha \in H_n(\Gamma;A)$ and let $\overline \alpha :=
  H_n(\pi;A)(\alpha)$ be the push-forward of~$\alpha$. We will now
  produce cycles representing~$\alpha$ whose~$\ell^1$-norm
  approximates~$\|\overline\alpha\|_1$.  As input for our
  construction, we need the following:
  \begin{itemize}
  \item A cycle~$c \in C_n(\Gamma;A)$ representing~$\alpha$;
    let $\overline c := C_n(\pi;A)(c)$.
  \item For every~$m \in \N_{>0}$ a cycle~$z_m \in C_n(\Gamma/N;A_N)$
    that represents~$\overline \alpha$ and satisfies
    \[ |z_m|_1 \leq \|\overline \alpha\|_1 + \frac1m.
    \]
  \item For every~$m \in \N_{>0}$ a chain~$\overline b_m \in C_{n+1}(\Gamma/N;A_N)$
    with
    \[ \partial \overline b_m = z_m - \overline c;
    \]
    let $b_m \in C_{n+1}(\Gamma;A)$ be some $\pi$-lift of~$\overline
    b_m$.
  \item A F\o lner sequence~$(F_k)_{k \in \N}$ for~$N$; let
    $(\psi_k)_{k \in \N}$ the corresponding sequence of averaging maps
    as in Theorem~\ref{mainthmcoeff}.
  \end{itemize}
  For~$k,m \in \N$ we then set
  \[ c_{k,m} := (\id_A \otimes_\Gamma \psi_k)(c + \partial b_m)
     \in C_n(\Gamma;A).
  \]
  From Theorem~\ref{mainthmcoeff}, we obtain for every~$m \in \N$ that
  \begin{align*}
    \lim_{k \rightarrow \infty} |c_{k,m}|_1
    = \bigl| C_n(\pi;A)(c + \partial b_m)\bigr|_1
    = | \overline z_k|_1
    \leq \|\overline\alpha\|_1 + \frac1m.
  \end{align*}
  Moreover, Proposition~\ref{prop:normestimate} can also be used to get
  an explicit estimate on the rate of convergence in terms of the F\o lner
  sequence~$(F_k)_{k \in \N}$. 
\end{proof}

The situation is particularly simple if the push-forward satisfies
$\overline \alpha =0$ because we can then take~$z_m := 0$ for all~$m
\in \N$. So, in this case, we only need to find a cycle~$c$
representing~$\alpha$, a single filling of~$\overline c$, and a F\o lner
sequence for~$N$.

\begin{rem}
  The same modifications as in Remark~\ref{rem:nonfingen} allow also to
  generalise Corollary~\ref{cor:isometriccoeff} and its explicit proof to
  the case that the normal subgroup is not finitely generated. 
\end{rem}

\subsection{Efficient cycles: spaces}\label{subsec:asph}

Using Proposition~\ref{prop:comp}, we can translate the explicit
constructions in the proof of Theorem~\ref{mainthmcoeff} and
Corollary~\ref{cor:isometriccoeff} to the case of aspherical spaces.
At this point it might be helpful to unravel these conversions and
spell out the action of tuples on singular simplices (as needed in
the averaging maps) in more explicit
terms:

Let $M$ be an aspherical space (with universal covering and deck
transformation action), let $\Gamma := \pi_1(M)$ be the fundamental
group of~$M$, and let $D \subset \ucov M$ be a set-theoretic
fundamental domain for the deck transformation action. If $\sigma \in
\map(\Delta^n,\ucov M)$ and $(\eta_0, \dots, \eta_n) \in
\Gamma^{n+1}$, then we construct $\sigma \cdot (\eta_0, \dots,
\eta_n)$ as follows:
\begin{itemize}
\item We determine the group elements~$\gamma_0, \dots, \gamma_n \in \Gamma$ 
  with~$\sigma(e_j) \in \Gamma_j \cdot D$.
\item We then look at the tuple~$(\gamma_0\cdot \eta_0, \dots,
  \gamma_n \cdot \eta_n) \in \Gamma^{n+1}$. 
\item Using the inductive filling procedure alluded to in
  Proposition~\ref{prop:comp}, we reconstruct a singular simplex
  on~$\ucov M$ from~$(\gamma_0 \cdot \eta_0, \dots, \gamma_n \cdot \eta_n)$.
  This is the desired singular simplex~$\sigma \cdot (\eta_0, \dots, \eta_n)$.
\end{itemize}
For a finite subset~$F \subset \Gamma$, the averaging map then has the form
\begin{align*}
  C_n(\ucov M;\R) & \longrightarrow C_n(\ucov M;\R) \\
  \map(\Delta^n,\ucov M) \ni \sigma
  & \longmapsto \frac1{|F|^{n+1}} \cdot \sum_{\eta \in F^{n+1}} \sigma \cdot \eta.
\end{align*}

In particular, we hence obtain explicit $\ell^1$-efficient fundamental
cycles of aspherical manifolds with (finitely generated) amenable
normal subgroup in terms of knowledge of cycles on the quotient and
F\o lner sequences of this normal subgroup.


\medskip
\vfill

\noindent
\emph{Clara L\"oh}\\[.5em]
  {\small
  \begin{tabular}{@{\qquad}l}
    Fakult\"at f\"ur Mathematik,
    Universit\"at Regensburg,
    93040 Regensburg\\
    \textsf{clara.loeh@mathematik.uni-r.de},\\
    \textsf{http://www.mathematik.uni-r.de/loeh}
  \end{tabular}}

\end{document}